\documentclass[a4paper,12pt,fleqn]{article}
\usepackage[latin1]{inputenc}
\usepackage[T1]{fontenc}
\usepackage[a4paper,lmargin=4.1cm,textwidth=12.8cm,tmargin=4.9cm,textheight=18.5cm]{geometry}
\usepackage[all,2cell,v2]{xy}
\usepackage[all]{xy}
\UseAllTwocells
\usepackage[square,numbers,sort]{natbib}
\usepackage{ifthen}
\usepackage{mathbbol}
\usepackage{savesym}
\usepackage{amsmath}
\usepackage{amssymb}
\savesymbol{iint}
\restoresymbol{TXF}{iint}
\usepackage[thmmarks,standard]{ntheorem}
\usepackage{microtype}
\usepackage{stmaryrd}
\usepackage[english]{babel}
\newboolean{PourEditeur}
\setboolean{PourEditeur}{false}
\ifthenelse{\boolean{PourEditeur}}{%
}{%
  \usepackage{hyperref}
}
\usepackage{cleveref}
\usepackage[toc,page]{appendix}

\theoremstyle{plain}
\crefformat{result}{result~#2#1#3}
\crefformat{proposition}{proposition~#2#1#3}
\crefformat{remark}{remark~#2#1#3}
\newcommand*{\G}{\ensuremath{\mathbb{G}\text{r}}}

\newcommand*{\OG}{\ensuremath{\omega\text{-}\G}}
\newcommand*{\C}{\ensuremath{\mathbb{C}\mathbb{A}T}}
\newcommand*{\TC}{\ensuremath{\mathbb{T}\text{-}\C}}
\newcommand*{\TG}{\ensuremath{\mathbb{T}\text{-}\mathbb{G}\text{r}}}

\newcommand*{\T}{\ensuremath{\mathbb{T}}}

\title{Notes on $\omega$-graphs, reflexive $\omega$-graphs, their higher transformations, and $\omega$-operads}
\author{Camell Kachour}
\begin{document}
\maketitle
\begin{abstract}
In this note we propose an $\omega$-operadical way to prove the existence of 
the $\omega$-graph of the $\omega$-graphs and the reflexive $\omega$-graph 
of the reflexive $\omega$-graphs.
\end{abstract}

\begin{minipage}{118mm}{\small
}\end{minipage}

\tableofcontents
\vspace{1cm}

\section*{Introduction}
\label{intro}

In \cite{camell:coend} we have proposed a unified technology to define the $\omega$-magma of the
$\omega$-magmas and the reflexive $\omega$-magma of the reflexive $\omega$-magmas.
Especially we have conjectured that, up to the contractibility of
some specific $\omega$-operads of coendomorphisms, this technology can be used to define the strict $\omega$-category
of the strict $\omega$-categories, but also the weak $\omega$-category
of the weak $\omega$-categories. In \cite{camell:coend} this technology uses the central notion
of the standard action of the higher transformations\footnote{That we call standard action for short.}. However it is very important to notice that the discovery of this technology of the standard actions\footnote{Which in fact can be generalised as we will see in a future paper in progress},
is completely independent with the notion of contractibility for the higher transformations that we have proposed in \cite{camell:coend}: In particular it is the contiguity of these two ideas which, we conjectured, gives an $\omega$-operadical approach of the weak $\omega$-category of the weak $\omega$-categories.

In this note we just use this technology of the standard action to prove the existence of the $\omega$-graph of the $\omega$-graphs, and the reflexive $\omega$-graph of the reflexive $\omega$-graphs, where in that case no contractibility
notion are involved. In fact this article can be considered as a result of \cite{camell:coend}, and has the goal to bring in light the power of this technology of
the standard action, for these basic and simple higher structures which are respectively $\omega$-graphs and
 reflexive $\omega$-graphs. In addition the author wishes to convice the reader about the relevance of the technology developed in \cite{camell:coend} for higher category theory with $\omega$-operads.

A key ingredient for our purpose occur in the level of the pointed $\T$-graphs (see \cite{camell:coend}): We use a
coglobular complex in $\TG_{p,c}$

\[\xymatrix{G^{0}\ar[rr]<+2pt>^{\delta^{1}_{0}}\ar[rr]<-2pt>_{\kappa^{1}_{0}}
  &&G^{1}\ar[rr]<+2pt>^{\delta^{1}_{2}}\ar[rr]<-2pt>_{\kappa^{2}_{1}}
  &&G^{2}\ar@{.>}[r]<+2pt>^{}\ar@{.>}[r]<-2pt>_{}
  &G^{n-1}\ar[rr]<+2pt>^{\delta^{n}_{n-1}}\ar[rr]<-2pt>_{\kappa^{n}_{n-1}}
  &&G^{n}\ar@{.}[r]<+2pt>\ar@{.}[r]<-2pt>&}\]
which is more basic than the coglobular complex $C^{\bullet}$ in $\TG_{p,c}$

\[\xymatrix{C^{0}\ar[rr]<+2pt>^{\delta^{1}_{0}}\ar[rr]<-2pt>_{\kappa^{1}_{0}}
  &&C^{1}\ar[rr]<+2pt>^{\delta^{1}_{2}}\ar[rr]<-2pt>_{\kappa^{2}_{1}}
  &&C^{2}\ar@{.>}[r]<+2pt>^{}\ar@{.>}[r]<-2pt>_{}
  &C^{n-1}\ar[rr]<+2pt>^{\delta^{n}_{n-1}}\ar[rr]<-2pt>_{\kappa^{n}_{n-1}}
  &&C^{n}\ar@{.}[r]<+2pt>\ar@{.}[r]<-2pt>&}\]
which is used in \cite{camell:coend}, and called the coglobular
     complex for the higher transformations. The coglobular complex $G^{\bullet}$ is build just by removing all cells "$\mu^{n}_{p}$" and "$\nu^{n}_{p}$" from it.

In this article we note $\mathbb{S}et$ the category of sets and $\mathbb{S}ET$ the
category of the sets and large sets.

  \vspace*{1cm}   

 \section{Standard actions of \texorpdfstring{$\omega$}{omega}-operads associated to a coglobular complex in \texorpdfstring{$\TC_c$}{TCc}}
  \label{short_intro}
 In this paragraph we recall briefly the section $2$ of the article \cite{camell:coend}, with the
 slight modification in the level of the combinatoric used in $\TG_{p,c}$ as in the last section. However
 it is necessary that the reader read also the section $1$ of this article
 to understand well this section.

 Consider a category of $\omega$-operads equipped with a specific structure
  or having a certain property that we call "$P$" such that if we call $P\TC_c$ this category, then the forgetful functor $U_{P}$ to $\TG_{p,c}$ has a left adjunction $F_{P}$.
 If we apply $F_{P}$ to the coglobular complex $G^{\bullet}$ (see the \textit{Introduction}) we obtain a coglobular complex in  $P\TC_c$

    \[\xymatrix{B^{0}_{P}\ar[rr]<+2pt>^{\delta^{1}_{0}}\ar[rr]<-2pt>_{\kappa^{1}_{0}}
  &&B^{1}_{P}\ar[rr]<+2pt>^{\delta^{1}_{2}}\ar[rr]<-2pt>_{\kappa^{2}_{1}}
  &&B^{2}_{P}\ar@{.>}[r]<+2pt>^{}\ar@{.>}[r]<-2pt>_{}
  &B^{n-1}_{P}\ar[rr]<+2pt>^{\delta^{n}_{n-1}}\ar[rr]<-2pt>_{\kappa^{n}_{n-1}}
  &&B^{n}_{P}\ar@{.}[r]<+2pt>\ar@{.}[r]<-2pt>&}\]
 which is also a coglobular object $W_{P}$ of $\TC_c$. Thus we obtain the resulting
  standard action\footnote{See paragraphs $1$ and $2$ of \cite{camell:coend} for
    a complete description of this diagram. Here $1$ design the terminal
    $\omega$-graph.} of $\TC_{1}$

   \[\xymatrix{Coend(W_{P})\ar[rr]^{Coend(\mathbb{A}lg(.))}&&
 Coend(A_{P}^{op})\ar[rr]^{Coend(Ob(.))}&&Coend(A_{0,P}^{op}) }\]
 where in particular $Coend(W_{P})$ is the monochromatic $\omega$-operad of coendomorphism associated to this
   coglobular complex. These kind of standard action are similar to those in \cite{camell:coend}, but much more simpler combinatoricaly speaking, because
   based on the more basic diagram $G^{\bullet}$ in $\TG_{p,c}$.

  As in \cite{camell:coend} the main problem is to build a morphism of $\omega$-operads between the monochromatic $\omega$-operad $B^{0}_{_{P}}$ (the "$0$-step"
   of the coglobular object $W_{P}$) and the monochromatic $\omega$-operad $Coend(W_{P})$ (build with the whole coglobular object $W_{P}$).
   If this morphism exist then it shows that $Coend(A_{0}^{op})$ is an algebra of $B^{0}_{P}$, and in that case we say that $B^{0}_{P}$ has the \textit{fractal property}.
   If $B^{0}_{P}$ has this fractal property, it means that $P$-$\omega$-categories, $P$-$\omega$-functors, $P$-$\omega$-natural transformations,
   $P$-$\omega$-modifications, etc. form a $B^{0}_{P}$-algebra.

   The paragraph below are
   devoted to give two examples of such $\omega$-operads with the fractal property. More
   precisely we are going to study the case where $P=Id$
   (i.e we just work with $\TC_c$; compare with the beginning of the section $4$ in \cite{camell:coend})
   and $P=Id_{u}$, i.e it is the property to have \textit{contractible units}
   (for a given $\omega$-operad; compare with the section $3.1$
   and the section $4$ of \cite{camell:coend}). In particular we are going to use the following
   free functors
  \[M: \xymatrix{\TG_{p,c}\ar[r]^{}&\TC_c,}\qquad Id_u: \xymatrix{\TG_{p,c}\ar[r]^{}&Id_u\TC_{c}}\]
   which are both defined in \cite{camell:coend}.

\section{The coglobular complexes of the graphical \texorpdfstring{$\omega$}{omega}-operads}
    \label{The_coglobular_complexes_of_the_graphical_omega_operads}
 The category $\OG$ of the $\omega$-graphs has canonical \textit{higher transformations}. First we are going to describe these higher transformations as presheaves
  on appropriate small categories $\mathbb{G}_{n}$, and then see that they form an $\omega$-graph that we call the $\omega$-graph of the $\omega$-graphs.
  It is the combinatoric description of these small categories $\mathbb{G}_{n}$ which allows to have a straightforward proof of the proposition \ref{proposition-graph},
  which basically says that these higher transformations are algebras for adapted $2$-coloured $\omega$-operads (see below).

 Consider the globe category $\mathbb{G}$

  \[\xymatrix{\bar{0}\ar[rr]<+2pt>^{s^{1}_{0}}\ar[rr]<-2pt>_{t^{1}_{0}}
   &&\bar{1}\ar[rr]<+2pt>^{s^{1}_{2}}\ar[rr]<-2pt>_{t^{2}_{1}}
   &&\bar{2}\ar@{.>}[r]<+2pt>^{}\ar@{.>}[r]<-2pt>_{}
   &\overline{n-1}\ar[rr]<+2pt>^{s^{n}_{n-1}}\ar[rr]<-2pt>_{t^{n}_{n-1}}
   &&\bar{n}\ar@{.}[r]<+2pt>\ar@{.}[r]<-2pt>&}\]
 subject to relations\footnote{Which are describe in the section $1$ of
  \cite{camell:coend}.} on cosources $s^{n+1}_{n}$ and cotargets $t^{n+1}_{n}$. For each each $n\geqslant 1$ we are going to build
 similar categories $\mathbb{G}_{n}$ in order that we will obtain a coglobular complex in $\C$

  \[\xymatrix{\mathbb{G}_{0}\ar[rr]<+2pt>^{\delta^{1}_{0}}\ar[rr]<-2pt>_{\kappa^{1}_{0}}
   &&\mathbb{G}_{1}\ar[rr]<+2pt>^{\delta^{1}_{2}}\ar[rr]<-2pt>_{\kappa^{2}_{1}}
   &&\mathbb{G}_{2}\ar@{.>}[r]<+2pt>^{}\ar@{.>}[r]<-2pt>_{}
   &\mathbb{G}_{n-1}\ar[rr]<+2pt>^{\delta^{n}_{n-1}}\ar[rr]<-2pt>_{\kappa^{n}_{n-1}}
   &&\mathbb{G}_{n}\ar@{.}[r]<+2pt>\ar@{.}[r]<-2pt>&}\]
 where $\mathbb{G}_{0}$ is just the globe category $\mathbb{G}$. Each category $\mathbb{G}_{n}$ is called the $n$-globe category. The
 $1$-globe category $\mathbb{G}_{1}$ is given by the category
 \[\xymatrix{\bar{0}\ar[rr]<+2pt>^{s^{1}_{0}}\ar[rr]<-2pt>_{t^{1}_{0}}
   &&\bar{1}\ar[rr]<+2pt>^{s^{2}_{1}}\ar[rr]<-2pt>_{t^{2}_{1}}
   &&\bar{2}\ar@{.>}[r]<+2pt>^{}\ar@{.>}[r]<-2pt>_{}
   &\overline{n-1}\ar[rr]<+2pt>^{s^{n}_{n-1}}\ar[rr]<-2pt>_{t^{n}_{n-1}}
   &&\bar{n}\ar@{.}[r]<+2pt>\ar@{.}[r]<-2pt>&\\
  \bar{0}'\ar[u]^{\alpha^{0}_{0}}\ar[rr]<+2pt>^{s'^{1}_{0}}\ar[rr]<-2pt>_{t'^{1}_{0}}
   &&\bar{1}'\ar[u]^{\alpha^{1}_{0}}\ar[rr]<+2pt>^{s'^{2}_{1}}\ar[rr]<-2pt>_{t'^{2}_{1}}
   &&\bar{2}'\ar[u]^{\alpha^{2}_{0}}\ar@{.>}[r]<+2pt>^{}\ar@{.>}[r]<-2pt>_{}
   &\overline{n-1}'\ar[u]^{\alpha^{n-1}_{0}}\ar[rr]<+2pt>^{s'^{n}_{n-1}}\ar[rr]<-2pt>_{t'^{n}_{n-1}}
   &&\bar{n}'\ar[u]^{\alpha^{n}_{0}}\ar@{.}[r]<+2pt>\ar@{.}[r]<-2pt>&
   }\]
 such that
 $\alpha^{n+1}_{0}\circ s'^{n+1}_{n}=s^{n+1}_{n}\circ \alpha^{n}_{0}$, $\alpha^{n+1}_{0}\circ t'^{n+1}_{n}=t^{n+1}_{n}\circ \alpha^{n}_{0}$.

 The $2$-globe category $\mathbb{G}_{2}$ is given by the category
 \[\xymatrix{\bar{0}\ar[rr]<+2pt>^{s^{1}_{0}}\ar[rr]<-2pt>_{t^{1}_{0}}
   &&\bar{1}\ar[rr]<+2pt>^{s^{1}_{2}}\ar[rr]<-2pt>_{t^{2}_{1}}
   &&\bar{2}\ar@{.>}[r]<+2pt>^{}\ar@{.>}[r]<-2pt>_{}
   &\overline{n-1}\ar[rr]<+2pt>^{s^{n}_{n-1}}\ar[rr]<-2pt>_{t^{n}_{n-1}}
   &&\bar{n}\ar@{.}[r]<+2pt>\ar@{.}[r]<-2pt>&\\
  \bar{0}'\ar[u]<+2pt>_{\alpha^{0}_{0}}\ar[u]<-2pt>^{\beta^{0}_{0}}\ar[rr]<+2pt>^(.4){s'^{1}_{0}}\ar[rr]<-2pt>_(.4){t'^{1}_{0}}
   &&\bar{1}'\ar[llu]_(.4){}\ar[u]<+2pt>_{\alpha^{1}_{0}}\ar[u]<-2pt>^{\beta^{1}_{0}}\ar[rr]<+2pt>^{s'^{2}_{1}}\ar[rr]<-2pt>_{t'^{2}_{1}}
   &&\bar{2}'\ar[u]<+2pt>_{\alpha^{n-1}_{0}}\ar[u]<-2pt>^{\beta^{n-1}_{0}}\ar@{.>}[r]<+2pt>^{}\ar@{.>}[r]<-2pt>_{}
   &\overline{n-1}'\ar[u]<+2pt>_{\alpha^{n-1}_{0}}\ar[u]<-2pt>^{\beta^{n-1}_{0}}\ar[rr]<+2pt>^{s'^{n}_{n-1}}\ar[rr]<-2pt>_{t'^{n}_{n-1}}
   &&\bar{n}'\ar[u]<+2pt>_{\alpha^{n}_{0}}\ar[u]<-2pt>^{\beta^{n}_{0}}\ar@{.}[r]<+2pt>\ar@{.}[r]<-2pt>&
   }\]
   where in particular we have an arrow $\xymatrix{\xi_{1}:\bar{1}'\ar[r]&\bar{0}}$. Arrows $s^{n+1}_{n}$, $t^{n+1}_{n}$, $\alpha^{n}_{0}$, $\beta^{n}_{0}$, and
 $\xi_{1}$ satisfy the following relations

 \begin{itemize}
    \item $\alpha^{n+1}_{0}\circ s'^{n+1}_{n}=s^{n+1}_{n}\circ \alpha^{n}_{0}$, $\alpha^{n+1}_{0}\circ t'^{n+1}_{n}=t^{n+1}_{n}\circ \alpha^{n}_{0}$,\\
    \item  $\beta^{n+1}_{0}\circ s'^{n+1}_{n}=s^{n+1}_{n}\circ \beta^{n}_{0}$, $\beta^{n+1}_{0}\circ t'^{n+1}_{n}=t^{n+1}_{n}\circ \beta^{n}_{0}$,\\
    \item $\xi_{1}\circ s'^{1}_{0}=\alpha^{0}_{0}$ and $\xi_{1}\circ t'^{1}_{0}=\beta^{0}_{0}$.
 \end{itemize}
 More generally the $n$-globe category $\mathbb{G}_{n}$ is given by the category

 \[\xymatrix{\bar{0}\ar[rr]<+2pt>^(.7){s^{1}_{0}}\ar[rr]<-2pt>_(.7){t^{1}_{0}}
   &&\bar{1}\ar[rr]<+2pt>^{s^{1}_{2}}\ar[rr]<-2pt>_{t^{2}_{1}}
   &&\bar{2}\ar@{.>}[r]<+2pt>^{}\ar@{.>}[r]<-2pt>_{}
   &\overline{n-1}\ar[rr]<+2pt>^{s^{n}_{n-1}}\ar[rr]<-2pt>_{t^{n}_{n-1}}
   &&\bar{n}\ar@{.}[r]<+2pt>\ar@{.}[r]<-2pt>&\\\\
  \bar{0}'\ar[uu]<+2pt>_{\alpha^{0}_{0}}\ar[uu]<-2pt>^{\beta^{0}_{0}}\ar[rr]<+2pt>^(.3){s'^{1}_{0}}\ar[rr]<-2pt>_(.3){t'^{1}_{0}}
   &&\bar{1}'\ar[lluu]<+2pt>\ar[lluu]<-2pt>\ar[uu]<+2pt>_(.3){\alpha^{1}_{0}}\ar[uu]<-2pt>^(.3){\beta^{1}_{0}}\ar[rr]<+2pt>^(.3){s'^{2}_{1}}\ar[rr]<-2pt>_(.3){t'^{2}_{1}}
   &&\bar{2}'\ar[lllluu]<+2pt>\ar[lllluu]<-2pt>\ar[uu]<+2pt>_(.7){\alpha^{2}_{0}}\ar[uu]<-2pt>^(.7){\beta^{2}_{0}}\ar@{.>}[r]<+2pt>^{}\ar@{.>}[r]<-2pt>_{}
   &\overline{n-1}'\ar[llllluu]\ar[uu]<+2pt>_{\alpha^{n-1}_{0}}\ar[uu]<-2pt>^{\beta^{n-1}_{0}}
   \ar[rr]<+2pt>^(.3){s'^{n}_{n-1}}\ar[rr]<-2pt>_(.3){t'^{n}_{n-1}}
   &&\bar{n}'\ar[uu]<+2pt>_{\alpha^{n}_{0}}\ar[uu]<-2pt>^{\beta^{n}_{0}}\ar@{.}[r]<+2pt>\ar@{.}[r]<-2pt>&
   }\]
 where in particular we have an arrow
 $\xymatrix{\xi_{n-1}:\overline{n-1}'\ar[r]&\bar{0}}$, and also for each
 $1\leqslant p\leqslant n-2$, we have arrows
 $\xymatrix{\bar{p}'\ar[r]<+2pt>^{\alpha_{p}}\ar[r]<-2pt>_{\beta_{p}} &\bar{0}}$.
 Arrows $s^{n+1}_{n}$, $t^{n+1}_{n}$, $\alpha^{n}_{0}$, $\beta^{n}_{0}$,
 $\alpha_{p}$, $\beta_{p}$, and $\xi_{n-1}$ satisfy the following relations

 \begin{itemize}
    \item $\alpha^{n+1}_{0}\circ s'^{n+1}_{n}=s^{n+1}_{n}\circ \alpha^{n}_{0}$, $\alpha^{n+1}_{0}\circ t'^{n+1}_{n}=t^{n+1}_{n}\circ \alpha^{n}_{0}$,\\
    \item $\beta^{n+1}_{0}\circ s'^{n+1}_{n}=s^{n+1}_{n}\circ \beta^{n}_{0}$, $\beta^{n+1}_{0}\circ t'^{n+1}_{n}=t^{n+1}_{n}\circ \beta^{n}_{0}$.\\
    \item $\alpha_{p}\circ s'^{p}_{p-1}=\beta_{p}\circ s'^{p}_{p-1}=\alpha_{p-1}$ and
  $\alpha_{p}\circ t'^{p}_{p-1}=\beta_{p}\circ t'^{p}_{p-1}=\alpha_{p-1}$, and we put $\alpha_{0}:=\alpha^{0}_{0}$
  and $\beta_{0}:=\beta^{0}_{0}$,\\
    \item $\xi_{n-1}\circ s'^{n-1}_{n-2}=\alpha_{n-2}$ and
  $\xi_{n-1}\circ t'^{n-1}_{n-2}=\beta_{n-2}$.
 \end{itemize}
   The cosources and cotargets functors
  $\xymatrix{\mathbb{G}_{0}\ar[r]<+2pt>^{\delta^{1}_{0}}\ar[r]<-2pt>_{\kappa^{1}_{0}}
   &\mathbb{G}_{1}}$ are such that $\delta^{1}_{0}$ sends $\mathbb{G}_{0}$ to $\mathbb{G}_{0}$, and $\kappa^{1}_{0}$ sends
   $\mathbb{G}_{0}$ to $\mathbb{G}'_{0}$. The cosources and cotargets functors
  $\xymatrix{\mathbb{G}_{1}\ar[r]<+2pt>^{\delta^{2}_{1}}\ar[r]<-2pt>_{\kappa^{2}_{1}}
   &\mathbb{G}_{2}}$ send $\mathbb{G}_{0}$ to $\mathbb{G}_{0}$, and $\mathbb{G}'_{0}$ to $\mathbb{G}'_{0}$.
   Also $\delta^{2}_{1}$ sends the symbols $\alpha^{n}_{0}$ to the symbols $\alpha^{n}_{0}$, and
  $\kappa^{2}_{1}$ sends the symbols $\alpha^{n}_{0}$ to the symbols $\beta^{n}_{0}$.

   Now consider the case $n\geqslant 3$. The cosources and cotargets functors
      \[\xymatrix{\mathbb{G}_{n-1}\ar[rr]<+2pt>^{\delta^{n}_{n-1}}\ar[rr]<-2pt>_{\kappa^{n}_{n-1}}
   &&\mathbb{G}_{n}}\]
 are build as follows : First we remove the cell $\xi_{n-1}$ and the cell $\beta_{n-2}$ from $\mathbb{G}_{n}$, and we obtain
   the category $\mathbb{G}^{-}_{n-1}$. Clearly we have an isomorphism of categories
    $\mathbb{G}^{-}_{n-1}\simeq \mathbb{G}_{n-1}$ (which sends $\alpha_{n-2}$ to $\xi_{n-2}$), and also the embedding
    $\xymatrix{\mathbb{G}^{-}_{n-1}\ar@{^{(}->}[r]^(.6){\delta'^{n-1}_{n}}&\mathbb{G}_{n}}$. The composition of this
    embedding with the last isomorphism gives $\xymatrix{\mathbb{G}_{n-1}\ar[r]^{\delta^{n}_{n-1}}&\mathbb{G}_{n}}$. The cotarget
    functor $\kappa^{n}_{n-1}$ is built similarly : First we remove the cell $\xi_{n-1}$ and the cell $\alpha_{n-2}$ from $\mathbb{G}_{n}$, and we obtain
   the category $\mathbb{G}^{+}_{n-1}$. Clearly we have an isomorphism of categories
    $\mathbb{G}^{+}_{n-1}\simeq \mathbb{G}_{n-1}$ (which sends $\beta_{n-2}$ to $\xi_{n-2}$), and also the embedding
    $\xymatrix{\mathbb{G}^{+}_{n-1}\ar@{^{(}->}[r]^(.6){\kappa'^{n}_{n-1}}&\mathbb{G}_{n}}$. The composition of this
    embedding with the last isomorphism gives $\xymatrix{\mathbb{G}_{n-1}\ar[r]^{\kappa^{n}_{n-1}}&\mathbb{G}_{n}}$.
    It is easy to see that these functors $\delta^{n}_{n-1}$ and $\kappa^{n}_{n-1}$ verify the cosourse cotarget conditions
    as for the globe category $\mathbb{G}_{0}$ above. When we applied the contravariant functor $[-;\mathbb{S}et]$
    to the following coglobular complex in $\C$

    \[\xymatrix{\mathbb{G}^{op}_{0}\ar[rr]<+2pt>^{\delta^{1}_{0}}\ar[rr]<-2pt>_{\kappa^{1}_{0}}
   &&\mathbb{G}^{op}_{1}\ar[rr]<+2pt>^{\delta^{1}_{2}}\ar[rr]<-2pt>_{\kappa^{2}_{1}}
   &&\mathbb{G}^{op}_{2}\ar@{.>}[r]<+2pt>^{}\ar@{.>}[r]<-2pt>_{}
   &\mathbb{G}^{op}_{n-1}\ar[rr]<+2pt>^{\delta^{n}_{n-1}}\ar[rr]<-2pt>_{\kappa^{n}_{n-1}}
   &&\mathbb{G}^{op}_{n}\ar@{.}[r]<+2pt>\ar@{.}[r]<-2pt>&}\]

 it is easy to see that we obtain the $\omega$-graph of the $\omega$-graphs\footnote{For each
  $n\in\mathbb{N}$,
  $[\mathbb{G}^{op}_{n};\mathbb{S}et](0)$ means the set of objects of the presheaf category
  $[\mathbb{G}^{op}_{n};\mathbb{S}et]$.}

   \[\xymatrix{\ar@{.>}[r]<+2pt>^{}\ar@{.>}[r]<-2pt>_{}
   &[\mathbb{G}^{op}_{n};\mathbb{S}et](0)\ar[r]<+2pt>^{\sigma^{n}_{n-1}}\ar[r]<-2pt>_{\beta^{n}_{n-1}}
   &[\mathbb{G}^{op}_{n-1};\mathbb{S}et](0)\ar@{.>}[r]<+2pt>^{}\ar@{.>}[r]<-2pt>_{}
   &[\mathbb{G}^{op}_{1};\mathbb{S}et](0)\ar[r]<+2pt>^{\sigma^{1}_{0}}\ar[r]<-2pt>_{\beta^{1}_{0}}
   &[\mathbb{G}^{op}_{0};\mathbb{S}et](0) }\]
 An object of the category of presheaves $[\mathbb{G}^{op}_{n};\mathbb{S}et]$ is called an $(n,\omega)$-graphs
 \footnote{Do not confuse the $(n,\omega)$-graphs with the $(\infty,n)$-graphs that we have defined in \cite{camell:groupoids}, which are completely different object. In
 \cite{camell:groupoids}, $(\infty,n)$-graphs are an important kind of $\omega$-graphs, which play a central role to define an algebraic approach of the
 weak $(\infty,n)$-categories.}. For instance, if $n\geqslant 3$, the source functor $\sigma^{n}_{n-1}$ is described as follow :
 If $\xymatrix{X:\mathbb{G}^{op}_{n}\ar[r]&\mathbb{S}et}$ is an
  $(n,\omega)$-graph, thus $\xymatrix{X(\xi_{n-1}):X(\bar{0})\ar[r]&X(\overline{n-1}')}$ is its underlying $(n-1)$-transformation, and
   $\xymatrix{\sigma^{n}_{n-1}(X)(\xi_{n-2}):X(\bar{0})\ar[r]&X(\overline{n-2}')}$ is the underlying $(n-2)$-transformation of $\sigma^{n}_{n-1}(X)$, and
   is defined by : $\sigma^{n}_{n-1}(X)(\xi_{n-2})=X(s'^{n}_{n-1})\circ X(\xi_{n-1})$. The target functors $\beta^{n}_{n-1}$ can be also describe easily.

    Now lets come back to the coglobular complex $G^{\bullet}$ in $\TG_{p,c}$ build in the \textit{Introduction}.
 If we apply to it the free functor $M: \xymatrix{\TG_{p,c}\ar[r]^{}&\TC_{c}}$ (see \ref{short_intro})
    we obtain a  coglobular complex in  $\TC_c$

      \[\xymatrix{B^{0}_{Id}\ar[rr]<+2pt>^{\delta^{1}_{0}}\ar[rr]<-2pt>_{\kappa^{1}_{0}}
   &&B^{1}_{Id}\ar[rr]<+2pt>^{\delta^{1}_{2}}\ar[rr]<-2pt>_{\kappa^{2}_{1}}
   &&B^{2}_{Id}\ar@{.>}[r]<+2pt>^{}\ar@{.>}[r]<-2pt>_{}
   &B^{n-1}_{Id}\ar[rr]<+2pt>^{\delta^{n}_{n-1}}\ar[rr]<-2pt>_{\kappa^{n}_{n-1}}
   &&B^{n}_{Id}\ar@{.}[r]<+2pt>\ar@{.}[r]<-2pt>&}\]
 which produces the following globular complex in $\mathbb{C}AT$

     \[\xymatrix{\ar@{.>}[r]<+2pt>^{}\ar@{.>}[r]<-2pt>_{}
   &\underline{B}_{Id}^{n}-\mathbb{A}lg\ar[r]<+2pt>^{\sigma^{n}_{n-1}}\ar[r]<-2pt>_{\beta^{n}_{n-1}}
   &\underline{B}_{Id}^{n-1}-\mathbb{A}lg\ar@{.>}[r]<+2pt>^{}\ar@{.>}[r]<-2pt>_{}
   &\underline{B}_{Id}^{1}-\mathbb{A}lg\ar[r]<+2pt>^{\sigma^{1}_{0}}\ar[r]<-2pt>_{\beta^{1}_{0}}
   &\underline{B}_{Id}^{0}-\mathbb{A}lg }\]
 and we have the easy proposition
 \begin{proposition}
 \label{proposition-graph}
   The category $\underline{B}_{Id}^{0}$-$\mathbb{A}lg$ is the category $[\mathbb{G}^{op}_{0};\mathbb{S}et]$ of the $\omega$-graphs,
    $\underline{B}_{Id}^{1}$-$\mathbb{A}lg$ is the category $[\mathbb{G}^{op}_{1};\mathbb{S}et]$ of the $(1,\omega)$-graphs, and for
    each integer $n\geqslant 2$,
    $\underline{B}_{Id}^{n}$-$\mathbb{A}lg$ is the category $[\mathbb{G}^{op}_{n};\mathbb{S}et]$ of the $(n,\omega)$-graphs.
 \end{proposition}
    Lets note $B^{\bullet}_{Id}$ this coglobular object in $\TC_{c}$. Its standard action is given by the following diagram in $\TC_1$

    \[\xymatrix{Coend(B^{\bullet}_{Id})
    \ar[rr]^{Coend(\mathbb{A}lg(.))}&&
  Coend(A_{Id}^{op})\ar[rr]^{Coend(Ob(.))}&&End(A_{0,Id}) }\]
 It is a standard action of the higher transformations specific to the basic $\omega$-graph structure. The monochromatic $\omega$-operad $Coend(B^{\bullet}_{Id})$ of coendomorphism plays a central role
  for $\omega$-graphs, and we call it the \textit{white operad}. Also it is straightforward that $B^{0}_{Id}$
  has the fractal property, because it is initial in the category $\TC_1$ of the category
  of $\omega$-operads, thus we have a unique morphism of $\omega$-operads
   \[\xymatrix{ B^{0}_{Id}\ar[rr]^{!_{Id}}&&Coend(B^{\bullet}_{Id}) }\]
 If we compose it with the standard action of the $\omega$-graphs

  \[\xymatrix{Coend(B^{\bullet}_{Id})
    \ar[rr]^{Coend(\mathbb{A}lg(.))}&&
  Coend(A_{Id}^{op})\ar[rr]^{Coend(Ob(.))}&&End(A_{0,Id}) }\]
 we obtain a morphism of $\omega$-operads

       \[\xymatrix{ B^{0}_{Id}\ar[rrr]^(.4){\mathfrak{G}}&&&End(A_{0,Id}) }\]
 which express an action of the $\omega$-operad  $B^{0}_{Id}$ of the
  $\omega$-graphs on the globular complex
  $B^{\bullet}_{Id}$-$\mathbb{A}lg(0)$ in $\mathbb{S}ET$
   of the $(n,\omega)$-graphs ($n\in\mathbb{N}$), and thus gives an
     $\omega$-graph structure on the $(n,\omega)$-graphs ($n\in\mathbb{N}$).

 \section{The coglobular complexes of the reflexive graphical \texorpdfstring{$\omega$}{omega}-operads}
    \label{The_coglobular_complexes_of_the_reflexive_graphical_omega_operads}
 By basing on the globe category $\mathbb{G}_{0}$ (see the section \ref{The_coglobular_complexes_of_the_graphical_omega_operads}),
 we build the reflexive globe category $\mathbb{G}_{0,\text{r}}$ as follow : For each $n\in\mathbb{N}$ we add in $\mathbb{G}_{0}$ the formal morphism
 $\xymatrix{\overline{n+1}\ar[rr]^{1^{n}_{n+1}}&&\bar{n}}$ such that
 $1^{n}_{n+1}\circ s^{n+1}_{n}=1^{n}_{n+1}\circ t^{n+1}_{n}=1_{\bar{n}}$. For each $0\leqslant p<n$ we denote
 $1^{p}_{n}:=1^{p}_{p+1}\circ 1^{p+1}_{p+2}\circ ... \circ 1^{n-1}_{n}$.

 For each each $n\geqslant 1$ we are going to build similar categories $\mathbb{G}_{n,\text{r}}$ in order that we will
 obtain a coglobular complex in $\C$

 \[\xymatrix{\mathbb{G}_{0,\text{r}}\ar[rr]<+2pt>^{\delta^{1}_{0}}\ar[rr]<-2pt>_{\kappa^{1}_{0}}
  &&\mathbb{G}_{1,\text{r}}\ar@/_2pc/[ll]_{i^{0}_{1}}\ar[rr]<+2pt>^{\delta^{1}_{2}}\ar[rr]<-2pt>_{\kappa^{2}_{1}}
  &&\mathbb{G}_{2,\text{r}}\ar@/_2pc/[ll]_{i^{1}_{2}}\ar@{.>}[r]<+2pt>^{}\ar@{.>}[r]<-2pt>_{}
  &\mathbb{G}_{n-1,\text{r}}\ar[rr]<+2pt>^{\delta^{n}_{n-1}}\ar[rr]<-2pt>_{\kappa^{n}_{n-1}}
  &&\mathbb{G}_{n,\text{r}}\ar@/_2pc/[ll]_{i^{n-1}_{n}}\ar@{.}[r]<+2pt>\ar@{.}[r]<-2pt>&}\]
  equipped with coreflexivity functors $i^{n}_{n+1}$.
 Each category $\mathbb{G}_{n,\text{r}}$ is called the reflexive $n$-globe category. It is build
 as the categories $\mathbb{G}_n$ $(n\geqslant 1)$ where we just replace $\mathbb{G}_0$ and $\mathbb{G}_0'$ by
 $\mathbb{G}_{0,r}$ and $\mathbb{G}_{0,r}'$.

For each $n\geqslant 1$, the cosources and the cotargets functors
     \[\xymatrix{\mathbb{G}_{n-1,\text{r}}\ar[rr]<+2pt>^{\delta^{n}_{n-1}}\ar[rr]<-2pt>_{\kappa^{n}_{n-1}}
  &&\mathbb{G}_{n,\text{r}}}\]
are build as those
 \[\xymatrix{\mathbb{G}_{n-1}\ar[rr]<+2pt>^{\delta^{n}_{n-1}}\ar[rr]<-2pt>_{\kappa^{n}_{n-1}}
  &&\mathbb{G}_{n}}\]
of the section \ref{The_coglobular_complexes_of_the_graphical_omega_operads}, where in addition $\delta^{1}_{0}$ sends
for all $p\geqslant 0$, the reflexivity morphism $1^{p}_{p+1}$ to the reflexivity morphism $1^{p}_{p+1}$, and
$\kappa^{0}_{1}$ sends the reflexivity morphism $1^{p}_{p+1}$ to the reflexivity morphism $1'^{p}_{p+1}$.
Also, if $n\geqslant 2$, $\delta^{n}_{n+1}$ and $\kappa^{n}_{n+1}$ send for all $p\geqslant 0$, the reflexivity morphism $1^{p}_{p+1}$
to the reflexivity morphism $1^{p}_{p+1}$, and the reflexivity morphism $1'^{p}_{p+1}$ to the reflexivity morphism $1'^{p}_{p+1}$.
It is trivial to see that these functors $\delta^{n}_{n-1}$ and $\kappa^{n}_{n-1}$ verify the cosourse and cotarget conditions.

For each $n\geqslant 1$, the coreflexivity functor
     \[\xymatrix{\mathbb{G}_{n,\text{r}}\ar[rr]^{i^{n-1}_{n}}&&\mathbb{G}_{n-1,\text{r}}}\]
 is built as follow: the coreflexivity functor $\iota^{0}_{1}$ sends, for all $q\geqslant 0$, the object $\bar{q}$ to $\bar{q}$,
  the object $\bar{q}'$ to $\bar{q}$, the cosource morphisms $s^{q+1}_{q}$ and $s'^{q+1}_{q}$ to $s^{q+1}_{q}$, the
  cotarget morphisms $t^{q+1}_{q}$ and $t'^{q+1}_{q}$ to $t^{q+1}_{q}$, the functor morphisms $\alpha^{q}_{0}$ to $1_{\bar{q}}$.
Also the coreflexivity functor $i^{1}_{2}$ sends, for all $q\geqslant 0$, the object $\bar{q}$ to $\bar{q}$,
  the object $\bar{q}'$ to $\bar{q}'$, the cosource morphism $s^{q+1}_{q}$ to the cosource morphism $s^{q+1}_{q}$,
  the cosource morphism $s'^{q+1}_{q}$ to the cosource morphism $s'^{q+1}_{q}$,
  the cotarget morphism $t^{q+1}_{q}$ to the cotarget morphism $t^{q+1}_{q}$,
  the cotarget morphism $t'^{q+1}_{q}$ to the cotarget morphism $t'^{q+1}_{q}$,
  the functor morphisms $\alpha^{q}_{0}$ to the functor morphisms $\alpha^{q}_{0}$,
  the functor morphisms $\beta^{q}_{0}$ to the functor morphisms $\beta^{q}_{0}$, the
  natural transformation morphism $\xi_{1}$ to $\alpha^{0}_{0}\circ 1'^{0}_{1}$.

  Also for each $n\geqslant 3$, the coreflexivity functor $i^{n-1}_{n}$ sends, for all $q\geqslant 0$, the object $\bar{q}$ to $\bar{q}$,
  the object $\bar{q}'$ to $\bar{q}'$, the cosource morphism $s^{q+1}_{q}$ to the cosource morphism $s^{q+1}_{q}$,
  the cosource morphism $s'^{q+1}_{q}$ to the cosource morphism $s'^{q+1}_{q}$,
  the cotarget morphism $t^{q+1}_{q}$ to the cotarget morphism $t^{q+1}_{q}$,
  the cotarget morphism $t'^{q+1}_{q}$ to the cotarget morphism $t'^{q+1}_{q}$,
  the functor morphisms $\alpha^{q}_{0}$ to the functor morphisms $\alpha^{q}_{0}$,
  the functor morphisms $\beta^{q}_{0}$ to the functor morphisms $\beta^{q}_{0}$.
  Also if $0\leqslant p\leqslant n-3$, it sends the $p$-transformation $\alpha_{p}$ to
  the $p$-transformation $\alpha_{p}$, the $p$-transformation $\beta_{p}$ to
  the $p$-transformation $\beta_{p}$\footnote{By convention we put $\alpha_0=\alpha^{0}_{0}$ and $\beta_0=\beta^{0}_{0}$.
  In fact this convention is natural because in our point of view of $n$-transformations, $1$-transformations
  are the usual natural transformations, and a $0$-transformation must be seen a the underlying function
  $F_0$ acting on the $0$-cells of a functor $F$.},
  the $(n-2)$-transformation $\alpha_{n-2}$ and $\beta_{n-2}$ to the $(n-2)$-transformation $\xi_{n-2}$,
  and finally the $(n-1)$-transformation $\xi_{n-1}$ to $\xi_{n-2}\circ 1'^{n-2}_{n-1}$.

  With this construction it is not difficult to show that functors $i^{n-1}_{n}$ ($n\geqslant 1$) verify the coreflexivity identities
     \[i^{n-1}_{n}\circ \delta^{n}_{n-1}=1_{\mathbb{G}^{\text{r}}_{n-1}}=i^{n-1}_{n}\circ \kappa^{n}_{n-1}\]
 When we applied the contravariant functor $[-;\mathbb{S}et]$
   to the following coglobular complex in $\C$

   \[\xymatrix{\mathbb{G}^{op}_{0,\text{r}}\ar[rr]<+2pt>^{\delta^{1}_{0}}\ar[rr]<-2pt>_{\kappa^{1}_{0}}
  &&\mathbb{G}^{op}_{1,\text{r}}\ar@/_2pc/[ll]_{i^{0}_{1}}\ar@/_2pc/[ll]_{i}\ar[rr]<+2pt>^{\delta^{1}_{2}}\ar[rr]<-2pt>_{\kappa^{2}_{1}}
  &&\mathbb{G}^{op}_{2,\text{r}}\ar@/_2pc/[ll]_{i^{1}_{2}}\ar@{.>}[r]<+2pt>^{}\ar@{.>}[r]<-2pt>_{}
  &\mathbb{G}^{op}_{n-1,\text{r}}\ar[rr]<+2pt>^{\delta^{n}_{n-1}}\ar[rr]<-2pt>_{\kappa^{n}_{n-1}}
  &&\mathbb{G}^{op}_{n,\text{r}}\ar@/_2pc/[ll]_{i^{n-1}_{n}}\ar@{.}[r]<+2pt>\ar@{.}[r]<-2pt>&}\]
  it is easy to see that we obtain the reflexive $\omega$-graph of the reflexive $\omega$-graphs

  \[\xymatrix{\ar@{.>}[r]<+2pt>^{}\ar@{.>}[r]<-2pt>_{}
  &[\mathbb{G}^{op}_{n,\text{r}};\mathbb{S}et](0)\ar[r]<+2pt>^{\sigma^{n}_{n-1}}\ar[r]<-2pt>_{\beta^{n}_{n-1}}
  &[\mathbb{G}^{op}_{n-1,\text{r}};\mathbb{S}et](0)\ar@/_2pc/[l]_{\iota^{n-1}_{n}}\ar@{.>}[r]<+2pt>^{}\ar@{.>}[r]<-2pt>_{}
  &[\mathbb{G}^{op}_{1,\text{r}};\mathbb{S}et](0)\ar[r]<+2pt>^{\sigma^{1}_{0}}\ar[r]<-2pt>_{\beta^{1}_{0}}
  &[\mathbb{G}^{op}_{0,\text{r}};\mathbb{S}et](0)\ar@/_2pc/[l]_{\iota^{0}_{1}} }\]
An object of the category of presheaves $[\mathbb{G}^{op}_{n};\mathbb{S}et]$ is called a reflexive $(n,\omega)$-graphs. For instance, if $n\geqslant 3$,
the reflexivity functor $\iota^{n-1}_{n}$ can be described as follow : If $\xymatrix{X:\mathbb{G}^{op}_{n-1}\ar[r]&\mathbb{S}et}$ is an
 $(n-1,\omega)$-graph, thus $\xymatrix{X(\xi_{n-2}):X(\bar{0})\ar[r]&X(\overline{n-2}')}$ is its underlying $(n-2)$-transformation, and
  $\xymatrix{\iota^{n-1}_{n}(X)(\xi_{n-1}):X(\bar{0})\ar[r]&X(\overline{n-1}')}$ is the $(n-1)$-transformation defined by :
  $\iota^{n-1}_{n}(X)(\xi_{n-1})=X(1'^{n-2}_{n-1})\circ X(\xi_{n-2})$.

Now consider the coglobular complex $G^{\bullet}$ in $\TG_{p,c}$ as in the \textit{Introduction}.
 If we apply the free functor $Id_u: \xymatrix{\TG_{p,c}\ar[r]^{}&Id_u\TC_{c}}$ (see \ref{short_intro})
  to it we obtain a coglobular complex in $\TC_c$

     \[\xymatrix{B^{0}_{Id_{u}}\ar[rr]<+2pt>^{\delta^{1}_{0}}\ar[rr]<-2pt>_{\kappa^{1}_{0}}
  &&B^{1}_{Id_{u}}\ar[rr]<+2pt>^{\delta^{1}_{2}}\ar[rr]<-2pt>_{\kappa^{2}_{1}}
  &&B^{2}_{Id_{u}}\ar@{.>}[r]<+2pt>^{}\ar@{.>}[r]<-2pt>_{}
  &B^{n-1}_{Id_{u}}\ar[rr]<+2pt>^{\delta^{n}_{n-1}}\ar[rr]<-2pt>_{\kappa^{n}_{n-1}}
  &&B^{n}_{Id_{u}}\ar@{.}[r]<+2pt>\ar@{.}[r]<-2pt>&}\]
   which produces the following globular complex in $\mathbb{C}AT$

    \[\xymatrix{\ar@{.>}[r]<+2pt>^{}\ar@{.>}[r]<-2pt>_{}
  &\underline{B}_{Id_{u}}^{n}-\mathbb{A}lg\ar[r]<+2pt>^{\sigma^{n}_{n-1}}\ar[r]<-2pt>_{\beta^{n}_{n-1}}
  &\underline{B}_{Id_{u}}^{n-1}-\mathbb{A}lg\ar@{.>}[r]<+2pt>^{}\ar@{.>}[r]<-2pt>_{}
  &\underline{B}_{Id_{u}}^{1}-\mathbb{A}lg\ar[r]<+2pt>^{\sigma^{1}_{0}}\ar[r]<-2pt>_{\beta^{1}_{0}}
  &\underline{B}_{Id_{u}}^{0}-\mathbb{A}lg }\]
and we have the easy proposition
\begin{proposition}
  The category $\underline{B}_{Id_{u}}^{0}$-$\mathbb{A}lg$ is the category $[\mathbb{G}^{op}_{0,\text{r}};\mathbb{S}et]$ of the reflexive $\omega$-graphs,
   $\underline{B}_{Id_{u}}^{1}$-$\mathbb{A}lg$ is the category $[\mathbb{G}^{op}_{1,\text{r}};\mathbb{S}et]$ of the reflexive $(1,\omega)$-graphs, and for
   each integer $n\geqslant 2$,
   $\underline{B}_{Id_{u}}^{n}$-$\mathbb{A}lg$ is the category $[\mathbb{G}^{op}_{n,\text{r}};\mathbb{S}et]$ of the reflexive $(n,\omega)$-graphs.
\end{proposition}
 Lets note $B^{\bullet}_{Id_{u}}$ this coglobular object in $\TC_{c}$.
   According to the section \ref{short_intro},
    we obtain the following diagram in $\TC_1$

   \[\xymatrix{Coend(B^{\bullet}_{Id_{u}})\ar[rr]^{Coend(\mathbb{A}lg(.))}&&
 Coend(A_{Id_{u}}^{op})\ar[rr]^{Coend(Ob(.))}&&Coend(A_{0,Id_{u}}^{op}) }\]
that we call the standard action of the reflexive $\omega$-graphs, thus which is a specific
 standard action.  The monochromatic $\omega$-operad $Coend(B^{\bullet}_{Id_{u}})$ of coendomorphism plays a central role
   for reflexive $\omega$-graphs, and we call it the \textit{blue operad}. Also we have the following proposition
\begin{proposition}
\label{proposition-fractality-BRG}
$B^{0}_{Id_{u}}$ has the fractal property.
\end{proposition}
\begin{proof}
  The units of the $\omega$-operad  $Coend(B^{\bullet}_{Id_{u}})$ are given by identity morphisms $\xymatrix{B^{n}_{Id_{u}}\ar[rr]^{1_{B^{n}_{Id_{u}}}}&&B^{n}_{Id_{u}}}$. We are going to exhibit a morphism of $\omega$-operads
   \[\xymatrix{B^{n+1}_{Id_{u}}\ar[rrrr]^{[1_{B^{n}_{Id_{u}}};1_{B^{n}_{Id_{u}}}]^{n}_{n+1}}&&&&B^{n}_{Id_{u}}}\]
which is the contractibility of the unit $1_{B^{n}_{Id_{u}}}$ with itself.

First consider the morphism of $\xymatrix{G^{n+1}\ar[rr]^{c^{n}_{n+1}}&&B^{n}_{Id_{u}}}$ of $\TG_c$, which sends $u_{m}$ to $u_{m}$,
$v_{m}$ to $v_{m}$, $\alpha^{m}_{0}$ to $\alpha^{m}_{0}$, $\beta^{m}_{0}$ to $\beta^{m}_{0}$, $\alpha_{p}$ to $\alpha_{p}$,
$\beta_{p}$ to $\beta_{p}$, $\alpha_{n}$ to $\xi_{n}$, $\beta_{n}$ to $\xi_{n}$, and $\xi_{n+1}$ to $\gamma([u_{n};u_{n}]^{n}_{n+1};1_{\xi_{n}})$. This
map $c^{n}_{n+1}$ equipped $B^{n}_{Id_{u}}$ with an operation system of the type $G^{n+1}$ and $B^{n}_{Id_{u}}$ has contractible units, so
by the universality of the map $\eta_{n+1}$, we get a unique morphism of $\omega$-operads $[1_{B^{n}_{Id_{u}}};1_{B^{n}_{Id_{u}}}]^{n}_{n+1}$

  \[\xymatrix{B^{n+1}_{Id_{u}}\ar@{.>}[rrrr]^{[1_{B^{n}_{Id_{u}}};1_{B^{n}_{Id_{u}}}]^{n}_{n+1}}&&&&B^{n}_{Id_{u}}\\
  G^{n+1}\ar[u]^{\eta_{n+1}}\ar[rrrru]_{c^{n}_{n+1}}}\]
This $(n+1)$-cell $[1_{B^{n}_{Id_{u}}};1_{B^{n}_{Id_{u}}}]^{n}_{n+1}$ has arity the degenerate tree $1^{n}_{n+1}([n])$.
Now we just need to prove that the following diagram commute serially, which shows that source and target of $[1_{B^{n}_{Id_{u}}};1_{B^{n}_{Id_{u}}}]^{n}_{n+1}$
is the unit $1_{B^{n}_{Id_{u}}}$

   \[\xymatrix{B^{n+1}_{Id_{u}}\ar[rrrrd]^{[1_{B^{n}_{Id_{u}}};1_{B^{n}_{Id_{u}}}]^{n}_{n+1}}\\
   B^{n}_{Id_{u}}\ar[u]<+2pt>^{\delta^{n+1}_{n}}\ar[u]<-2pt>_{\kappa^{n+1}_{n}}
   \ar[rrrr]_{1_{B^{n}_{Id_{u}}}}&&&&B^{n}_{Id_{u}}}\]
But we have the following diagram which, on the left side commute serially, and on the right side commute

   \[\xymatrix{B^{n}_{Id_{u}}\ar[rr]<+2pt>^{\delta^{n+1}_{n}}\ar[rr]<-2pt>_{\kappa^{n+1}_{n}}&&B^{n+1}_{Id_{u}}
 \ar[rrr]^{[1_{B^{n}_{Id_{u}}};1_{B^{n}_{Id_{u}}}]^{n}_{n+1}}&&&B^{n}_{Id_{u}}\\
 G^{n}\ar[u]^{\eta_{n}}\ar[rr]<+2pt>^{\delta^{n+1}_{n}}\ar[rr]<-2pt>_{\kappa^{n+1}_{n}}&&G^{n+1}\ar[u]_{\eta_{n+1}}\ar[rrru]_{c^{n}_{n+1}}}\]
The morphism $c^{n}_{n+1}$ is a morphism of $\TG_{p,c}$, and also the morphisms $\delta^{n+1}_{n}$ and $\kappa^{n+1}_{n}$ on the
bottom of this diagram. Their combinatorial descriptions show easily that we have the equalities
$c^{n}_{n+1}\circ\delta^{n+1}_{n}=c^{n}_{n+1}\circ\kappa^{n+1}_{n}=\eta_{n}$. So we have the equalities
$[1_{B^{n}_{Id_{u}}};1_{B^{n}_{Id_{u}}}]^{n}_{n+1}\circ\delta^{n+1}_{n}=[1_{B^{n}_{Id_{u}}};1_{B^{n}_{Id_{u}}}]^{n}_{n+1}\circ\kappa^{n+1}_{n}=1_{B^{n}_{Id_{u}}}$.
It show that the $\omega$-operad  $Coend(B^{\bullet}_{Id_{u}})$ has contractible units, and thus we have a unique morphism of $\omega$-operads
  \[\xymatrix{B^{0}_{Id_{u}}\ar[rr]^{!_{Id_{u}}}&&Coend(B^{\bullet}_{Id_{u}}) }\]
which express the fractality of $B^{0}_{Id_{u}}$.
\end{proof}
 If we compose the morphism $!_{Id_{u}}$ with the standard action of the reflexive $\omega$-graphs
 \[\xymatrix{Coend(B^{\bullet}_{Id_{u}})
   \ar[rr]^{Coend(\mathbb{A}lg(.))}&&
 Coend(A_{Id_{u}}^{op})\ar[rr]^{Coend(Ob(.))}&&Coend(A_{0,Id_{u}}){op} }\]
we obtain a morphism of $\omega$-operads

      \[\xymatrix{ B^{0}_{Id_{u}}\ar[rrr]^(.4){\mathfrak{G}_r}&&& End(A_{0,Id_{u}}) }\]
which express an action of the $\omega$-operad  $B^{0}_{Id_{u}}$ of the
 reflexive $\omega$-graphs on the globular complex
 $B^{\bullet}_{Id_{u}}$-$\mathbb{A}lg(0)$ in $\mathbb{S}ET$
  of the reflexive $(n,\omega)$-graphs ($n\in\mathbb{N}$), and thus gives a reflexive
    $\omega$-graph structure on the reflexive $(n,\omega)$-graphs ($n\in\mathbb{N}$).

\vspace{1cm}

  \bigbreak{}
  \begin{minipage}{1.0\linewidth}
    Camell \textsc{Kachour}\\
    83 Boulevard du Temple, 93390, Clichy-sous-Bois, France\\
    Phone: 0033143512807\\
    Email:\href{mailto:camell.kachour@gmail.com}{\url{camell.kachour@gmail.com}}
  \end{minipage}

\end{document}